\documentclass[runningheads,orivec]{llncs}
\usepackage[T1]{fontenc}
\usepackage{graphicx}
\usepackage{graphicx,wrapfig}
\usepackage{lipsum}
\usepackage{aliascnt}
\usepackage{url}
\usepackage{mathtools}
\usepackage{changepage}
\usepackage{multirow}
\usepackage{wrapfig}
\usepackage{algorithm}
\usepackage[noend]{algpseudocode}
\usepackage{color}
\usepackage{amsmath}
\usepackage{epsfig}
\usepackage{enumerate}
\usepackage{amsfonts}
\usepackage{amssymb}
\usepackage{mathrsfs}
\usepackage{ifpdf}
\usepackage{indentfirst,relsize}
\usepackage{comment}
\usepackage[colorlinks=true]{hyperref}

\usepackage{setspace,graphicx}
\usepackage{latexsym}
\usepackage[all]{xy}
\usepackage[dvipsnames]{pstricks}

\usepackage{pst-grad} 
\usepackage{pst-plot} 
\usepackage[inline]{enumitem} 
\usepackage{wrapfig}

 \usepackage{array}

\newcommand{\remove}[1]{}

\newcounter{Case}[theorem]

\newcounter{Subcase}[theorem]

\newcounter{Subsubcase}[theorem]

\newcounter{Subsubsubcase}[theorem]

\newaliascnt{observation}{theorem}

\aliascntresetthe{observation}

\usepackage{tikz}

\usepackage[]{tikz}
\usetikzlibrary{shapes.geometric, positioning, arrows.meta, calc,fit,arrows.meta}
\usetikzlibrary{decorations}
\usetikzlibrary{decorations.text}

\usetikzlibrary{positioning}
\usepackage{graphicx}

\begin{document}
\title{Improved upper bounds on color reversal by local inversions\thanks{Supported by SERB/ANRF MATRICS grant MTR/2022/000692: \textit{Algorithmic study on hereditary graph properties}}}

\author{Kumud Singh Porte \and
R B Sandeep \orcidID{0000-0003-4383-1819} \and
Kamal Santra\orcidID{0009-0006-5997-1452}}
\authorrunning{K. S. Porte et al.}

\institute{Department of Computer Science and Engineering, \\ Indian Institute of Technology Dharwad, \\ Dharwad, 580011, Karnataka, India \\
\email{\{cs24dp012,sandeep,ra.kamal.santra\}@iitdh.ac.in}}
\maketitle              

\begin{abstract}
We study the problem of color reversal in bicolored graphs under local inversions. A \emph{bicoloration} of a graph $G=(V,E)$ is a mapping $\beta: V \to \{-1,1\}$. A \emph{local inversion} at a vertex $v \in V$ consists of reversing the colors of all neighbors of $v$ and replacing the subgraph induced by these neighbors with its complement, while leaving $v$ and the rest of $G$ unchanged. Sabidussi (Discrete Mathematics, 1987) showed that any bicolored graph on $n$ vertices without isolated vertices can be color-reversed (that is, all vertex colors flipped while preserving the underlying graph) in at most $6n+3$ local inversions, and that any bicolored graph can be transformed into another bicolored graph on the same underlying graph in at most $9n$ local inversions. We improve both bounds: we prove that the first task can be accomplished in at most $4n-3$ local inversions, and the second in at most $ \left \lfloor \frac{11n-3}{2} \right \rfloor$ local inversions. Furthermore, we show that for stars and complete graphs, color reversal can be performed with at most $3n$ local inversions.

\end{abstract}

\section{Introduction}

Local transformations of graphs form a common language across structural graph theory, algebraic graph theory, and quantum information. Among these, \emph{local complementation}---toggling adjacency within the open neighborhood of a chosen vertex—plays a central role. It appears in Bouchet’s theory of isotropic systems and the structure of circle graphs~\cite{DBLP:journals/jct/Bouchet88,DBLP:journals/jct/Bouchet94}, underpins the vertex–minor relation and rank-width~\cite{DBLP:journals/jct/Oum05}, and captures equivalences between quantum graph states under local Clifford operations~\cite{hein2004multiparty,van2004graphical}. Closely related, Seidel switching is well known for its applications to combinatorial and algebraic aspects of graphs~\cite{seidel1991survey}. Local complementation falls under the broad umbrella of graph modification problems, which have been extensively studied in algorithmic and parameterized complexity~\cite{DBLP:journals/dagstuhl-reports/BodlaenderHL14,DBLP:journals/ipl/Cai96,crespelle2023survey}.

We work with simple, undirected, labelled graphs $G=(V,E)$ whose vertices carry a binary color (encoded by $\beta:V\to\{-1,+1\}$). A single move, called \emph{local inversion}, at a vertex $v\in V$ applies local complementation at $v$ to the underlying graph and simultaneously flips the colors of all neighbors of $v$. Iterating moves along a string $w\in V^*$ transforms the bicolored graph $B=(G,\beta)$ into $B_w$, where $V^*$ denotes the set of all finite sequences of vertices from $V$. For $S\subseteq V$, we write $B^S$ for the bicolored graph that keeps the underlying graph $G$ and flips the colors of vertices in $S$ only; in particular, $B^{V}$ performs a \emph{global color reversal}. Our central quantity is the \emph{color reversal number} $cr(G)$: the minimum $\ell$ such that, for every initial coloring $\beta$, there exists a string $w$ with $|w|\le \ell$ and $B_w=B^{V}$. Note that if there exists a string $w$ such that $B_w = B^V$, then for every bicolored graph $B'$ on $G$, we also have $B'_w = (B')^V$. Therefore, it is independent of the initial coloring.

In 1987, Sabidussi~\cite{sabidussi1987color} introduced the color-reversal problem under local inversions. He proved that a bicolored graph without isolated vertices can be color-reversed using at most $6n+3$ local inversions. He also studied a more general problem: given two bicolored graphs $B$ and $B'$ with the same underlying graph $G$, can one transform $B$ into $B'$ using local inversions? He showed that, for any such pair, this can be done in at most $9n$ local inversions. Brijder and Hoogeboom~\cite{brijder2011group} applied these results in their study of the group structure of pivot and loop complementation on graphs and set systems.

\paragraph{Our contributions.}
We improve both bounds of Sabidussi~\cite{sabidussi1987color}.
\begin{itemize}
  \item For every connected $n$-vertex graph $G$ with $n \geq 2$,
  \[
    cr(G)\;\le\;
    \begin{cases}
      4n-4, & \text{if $n$ is even},\\
      4n-3, & \text{if $n$ is odd},
    \end{cases}
  \]
  obtained via parity-aware decompositions and repeated applications of \\constant-length strings (\autoref{thm:main}).
  \item For any two bicolored graphs $B=(G,\beta)$ and $B'=(G,\beta')$ on the same connected $n$-vertex graph $G$ with $n\geq 2$, there exists a string $w$ with
  $|w|\le \left \lfloor \tfrac{11n-3}{2} \right \rfloor$ such that $B_w = B'$ (\autoref{transform}).
  \item For star graphs and complete graphs $G$ with at least two vertices, we have $cr(G)\leq 3n$ (\autoref{star}, \autoref{complete}).
\end{itemize}

\autoref{thm:main} and \autoref{transform} crucially use the perfect forest theorem of Scott \\ ~\cite{scott2001induced}, which states that every graph of even order has a spanning forest in which each tree is an induced subgraph of the graph and every vertex has odd degree within its tree. See Caro et al.~\cite{caro2017two} for two shorter proofs of the theorem.  
All constructions that we provide are explicit and run in polynomial time.

The rest of the paper is organized as follows. \autoref{Section_2} introduces the preliminaries, providing definitions and notations for graphs, bicolored graphs, local complementation, and local inversion. It also states several known results and recalls the Perfect Forest Theorem. \autoref{Section_3} presents the proof of \autoref{thm:main}, which establishes an improved bound on $cr(G)$, and of \autoref{transform}, which bounds the number of moves required to transform one bicolored graph into another. \autoref{Section_4} focuses on special graph families, namely star graphs and complete graphs, and shows that stronger, family-specific bounds can be obtained through explicit constructions, leading to \autoref{star} and \autoref{complete}. Finally, \autoref{sec:conclusion} discusses a few open problems.

\section{Preliminaries}\label{Section_2}
We consider only simple, undirected, and labelled graphs. The vertex and edge sets of a graph $G$ are denoted by $V(G)$ and $E(G)$, respectively. Two graphs $G$ and $G'$ on the same vertex set are \emph{equivalent} if and only if $E(G)=E(G')$; equivalently, for all distinct $u,v\in V(G)$, $\{u,v\}\in E(G)$ if and only if $\{u,v\}\in E(G')$.
For a subset $S\subseteq V(G)$, we write $G[S]$ for the subgraph of $G$ induced by $S$. The \emph{open neighborhood} of a vertex $v$ in $G$ is denoted by $N_G(v)$.
A path, a complete graph, and a star on $t$ vertices are denoted by $P_t$, $K_t$, and $S_t$, respectively.

A \emph{bicoloration} of a graph $G=(V,E)$ is a mapping $\beta:V\to\{-1,1\}$.  
A \emph{bicolored graph} is a pair $B=(G,\beta)$, where $G$ is a graph and $\beta$ is a bicoloration of $G$.
Let $G=(V,E)$ be a graph and let $a\in V$. The \emph{local complement} of $G$ at $a$ is the graph $G_a=(V,E_a)$ obtained by toggling adjacency among the neighbors of $a$; formally, for distinct $x,y\in V$,
\[
\{x,y\}\in E_a \;\Longleftrightarrow\;
\begin{cases}
\{x,y\}\in E \text{ and } (x\notin N_G(a)\text{ or }y\notin N_G(a)),\\[4pt]
\text{or}\\[4pt]
\{x,y\}\notin E \text{ and } x,y\in N_G(a).
\end{cases}
\]
\autoref{fig:example4} illustrates the construction of $G_a$ from $G$.

\begin{figure}
\begin{center}
\begin{tikzpicture}
\node[draw=none, shape=circle, text width=2cm, name= p5,] {}; \node[draw=black, shape=circle, text width=2mm, name=e2] at(p5.90) {}; \node[draw=black, shape=circle, text width=2mm, name=e5] at(p5.20) {}; \node[draw=black, shape=circle, text width=2mm, name=e3] at(p5.160) {}; \node[draw=black, shape=circle, text width=2mm, name=e4] at(p5.200) {}; \node[draw=black, shape=circle, text width=2mm, name=e6] at(p5.340) {}; \node[, name=ei,left=5mm of e2] {};  \draw[-{Stealth[scale=1.5]}] (ei) to (e2); \draw[] (e2) to (e5); \draw[] (e3) to (e2); \draw[] (e3) to (e4); \draw[] (e3) to (e6); \draw[] (e5) to (e6); \draw[] (e4) to (e6); \draw[] (e2) to (e6); \draw[shift=(e2.center)] node[] {$a$}; \draw[shift=(e3.center)] node[] {$x$}; \draw[shift=(e4.center)] node[] {$x^{'}$}; \draw[shift=(e5.center)] node[] {$y$}; \draw[shift=(e6.center)] node[] {$y^{'}$}; \node[yshift=4mm] at(p5.270) {$G$}; \node[draw=none, shape=circle, text width=2cm, name= p6, right=2cm of p5] {}; \node[draw=black,shape=circle, text width=2mm, name=e2] at(p6.90) {}; \node[draw=black, shape=circle, text width=2mm, name=e5] at(p6.20) {}; \node[draw=black, shape=circle, text width=2mm, name=e3] at(p6.160) {}; \node[draw=black, shape=circle, text width=2mm, name=e4] at(p6.200) {}; \node[draw=black, shape=circle, text width=2mm, name=e6] at(p6.340) {}; \draw[] (e2) to (e5); \draw[] (e3) to (e2); \draw[] (e3) to (e4); \draw[] (e4) to (e6); \draw[] (e3) to (e5); \draw[] (e2) to (e6); \draw[shift=(e2.center)] node[] {$a$}; \draw[shift=(e3.center)] node[] {$x$}; \draw[shift=(e4.center)] node[] {$x^{'}$}; \draw[shift=(e5.center)] node[] {$y$}; \draw[shift=(e6.center)] node[] {$y^{'}$}; \node[yshift=4mm] at(p6.270) {$G_a$};
\end{tikzpicture}
\end{center}
\caption{Local complement of a graph $G$ with respect to vertex $a$.}
\label{fig:example4}
\end{figure}
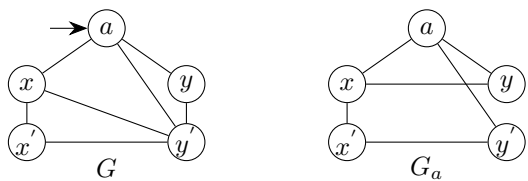

For a graph $G$, a \emph{string} on the alphabet $V(G)$ is a finite sequence of vertices of $V(G)$. The set $V(G)^*$ denotes the set of all finite sequences over $V(G)$. We use $\varepsilon$ to denote the \emph{empty string}.
Let $w\in V(G)^*$ be any string. We define the local complement of $G$ with respect to $w$, denoted $G_w$, inductively as follows:
\begin{itemize}
\item Base case: $G_\varepsilon=G$.
\item Inductive case: if $w=w'a$ with $a\in V(G)$, then $G_{w}= (G_{w'})_{a}$.
\end{itemize}

\autoref{fig:example5} shows an example of local complementation with respect to the string $w=abca$.

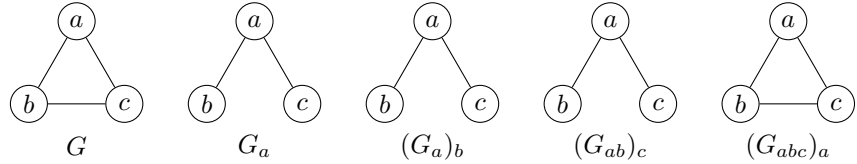
\begin{figure}
\begin{center}
\begin{tikzpicture}
\node[draw=none, shape=circle, text width=1.2cm, name= a1] {}; \node[draw=black, shape=circle, text width=2mm, name=a2] at(a1.90) {}; \node[draw=black, shape=circle, text width=2mm, name=a3] at(a1.210) {}; \node[draw=black, shape=circle, text width=2mm, name=a4] at(a1.330) {}; \draw[] (a2) to (a3); \draw[] (a4) to (a3); \draw[] (a2) to (a4); \draw[shift=(a2.center)] node[] {$a$}; \draw[shift=(a3.center)] node[] {$b$}; \draw[shift=(a4.center)] node[] {$c$}; 
\node[draw=none, shape=circle, text width=1.2cm, name= p2,right=9mm of a1] {}; \node[draw=black, shape=circle, text width=2mm, name=b2] at(p2.90) {}; \node[draw=black, shape=circle, text width=2mm, name=b3] at(p2.210) {}; \node[draw=black, shape=circle, text width=2mm, name=b4] at(p2.330) {}; \draw[] (b2) to (b3); \draw[] (b2) to (b4); \draw[shift=(b2.center)] node[] {$a$}; \draw[shift=(b3.center)] node[] {$b$}; \draw[shift=(b4.center)] node[] {$c$}; 
\node[draw=none, shape=circle, text width=1.2cm, name= p3,right=9mm of p2] {}; \node[draw=black, shape=circle, text width=2mm, name=c2] at(p3.90) {}; \node[draw=black, shape=circle, text width=2mm, name=c3] at(p3.210) {}; \node[draw=black, shape=circle, text width=2mm, name=c4] at(p3.330) {}; \draw[] (c2) to (c3); \draw[] (c2) to (c4); \draw[shift=(c2.center)] node[] {$a$}; \draw[shift=(c3.center)] node[] {$b$}; \draw[shift=(c4.center)] node[] {$c$}; 
\node[draw=none, shape=circle, text width=1.2cm, name= p4,right=9mm of p3] {}; \node[draw=black, shape=circle, text width=2mm, name=d2] at(p4.90) {}; \node[draw=black, shape=circle, text width=2mm, name=d3] at(p4.210) {}; \node[draw=black, shape=circle, text width=2mm, name=d4] at(p4.330) {}; \draw[] (d2) to (d3); \draw[] (d2) to (d4); \draw[shift=(d2.center)] node[] {$a$}; \draw[shift=(d3.center)] node[] {$b$}; \draw[shift=(d4.center)] node[] {$c$}; 
\node[draw=none, shape=circle, text width=1.2cm, name= p5,right=9mm of p4] {}; \node[draw=black, shape=circle, text width=2mm, name=e2] at(p5.90) {}; \node[draw=black, shape=circle, text width=2mm, name=e3] at(p5.210) {}; \node[draw=black, shape=circle, text width=2mm, name=e4] at(p5.330) {}; \draw[] (e2) to (e3); \draw[] (e2) to (e4); \draw[] (e3) to (e4); \draw[shift=(e2.center)] node[] {$a$}; \draw[shift=(e3.center)] node[] {$b$}; \draw[shift=(e4.center)] node[] {$c$}; \node[yshift=-2mm] at(a1.270) {$G$}; \node[yshift=-2mm] at(p2.270) {$G_a$}; \node[yshift=-2mm] at(p3.270) {$(G_a)_b$}; \node[yshift=-2mm] at(p4.270) {$(G_{ab})_c$}; \node[yshift=-2mm] at(p5.270) {$(G_{abc})_a$};
\end{tikzpicture}
\end{center}
\caption{Local complement of a graph with respect to the string $abca$.}
\label{fig:example5}
\end{figure}

Given a bicoloration $\beta$, the bicoloration with respect to a vertex $a$, denoted $\beta_a$, is defined by inverting the color of each neighbor of $a$:
\[
 \beta_{a}(x)=
\begin{cases}
-\beta(x),  & \text{if } x\in N_G(a),\\[4pt]
\beta(x),   & \text{otherwise}.
\end{cases}
\]

Let $B=(G,\beta)$ be a bicolored graph and let $a\in V(G)$. Writing $G_a$ and $\beta_a$ as above, the \emph{local inversion} of $B$ at $a$, denoted $B_a$, is $(G_a,\beta_a)$.
Let $B=(G,\beta)$ be a bicolored graph and let $w\in V(G)^*$ be a string. The local inversion of $B$ with respect to $w$ is defined inductively as follows:
\begin{itemize}
    \item Base case: $B_\varepsilon = B$.
    \item Inductive case: if $w=w'a$ with $a\in V(G)$, then $B_{w} = (B_{w'})_{a}$.
\end{itemize}
For example, see \autoref{fig:example8}.

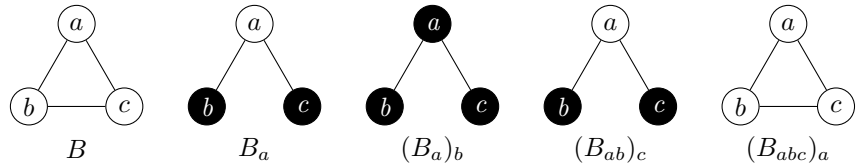
\begin{figure}
\begin{center}
\begin{tikzpicture}
\node[draw=none, shape=circle, text width=1.2cm, name= g1] {}; \node[draw=black, shape=circle, text width=2mm, name=a2] at(g1.90) {}; \node[draw=black, shape=circle, text width=2mm, name=a3] at(g1.210) {}; \node[draw=black, shape=circle, text width=2mm, name=a4] at(g1.330) {}; \draw[] (a2) to (a3); \draw[] (a4) to (a3); \draw[] (a2) to (a4); \draw[shift=(a2.center)] node[] {$a$}; \draw[shift=(a3.center)] node[] {$b$}; \draw[shift=(a4.center)] node[] {$c$}; 
\node[draw=none, shape=circle, text width=1.2cm, name= g2,right=9mm of g1] {}; \node[draw=black, shape=circle, text width=2mm, name=h2] at(g2.90) {}; \node[draw=black, fill=black, shape=circle, text width=2mm, name=h3] at(g2.210) {}; \node[draw=black, fill = black, shape=circle, text width=2mm, name=h4] at(g2.330) {}; \draw[] (h2) to (h3); \draw[] (h2) to (h4); \draw[shift=(h2.center)] node[] {$a$}; \draw[shift=(h3.center)] node[text=white] {$\mathbf{\textit{b}}$}; \draw[shift=(h4.center)] node[text=white] {$\mathbf{\textit{c}}$}; 
\node[draw=none, shape=circle, text width=1.2cm, name= g3,right=9mm of g2] {}; \node[draw=black, fill = black,shape=circle, text width=2mm, name=i2] at(g3.90) {}; \node[draw=black,fill = black, shape=circle, text width=2mm, name=i3] at(g3.210) {}; \node[draw=black, fill = black,shape=circle, text width=2mm, name=i4] at(g3.330) {}; \draw[] (i2) to (i3); \draw[] (i2) to (i4); \draw[shift=(i2.center)] node[text=white] {$\mathbf{\textit{a}}$}; \draw[shift=(i3.center)] node[text=white] {$\mathbf{\textit{b}}$}; \draw[shift=(i4.center)] node[text=white] {$\mathbf{\textit{c}}$}; 
\node[draw=none, shape=circle, text width=1.2cm, name= g4,right=9mm of g3] {}; \node[draw=black, fill = white,shape=circle, text width=2mm, name=j2] at(g4.90) {}; \node[draw=black, fill=black, shape=circle, text width=2mm, name=j3] at(g4.210) {}; \node[draw=black, fill=black, shape=circle, text width=2mm, name=j4] at(g4.330) {}; \draw[] (j2) to (j3); \draw[] (j2) to (j4); \draw[shift=(j2.center)] node[] {$a$}; \draw[shift=(j3.center)] node[text=white] {$\mathbf{\textit{b}}$}; \draw[shift=(j4.center)] node[text=white] {$\mathbf{\textit{c}}$}; 
\node[draw=none, shape=circle, text width=1.2cm, name= g5,right=9mm of g4] {}; \node[draw=black, shape=circle, text width=2mm, name=k2] at(g5.90) {}; \node[draw=black, shape=circle, text width=2mm, name=k3] at(g5.210) {}; \node[draw=black, shape=circle, text width=2mm, name=k4] at(g5.330) {}; \draw[] (k2) to (k3); \draw[] (k2) to (k4); \draw[] (k3) to (k4); \draw[shift=(k2.center)] node[] {$a$}; \draw[shift=(k3.center)] node[] {$b$}; \draw[shift=(k4.center)] node[] {$c$}; \node[yshift=-2mm] at(g1.270) {$B$}; \node[yshift=-2mm] at(g2.270) {$B_a$}; \node[yshift=-2mm] at(g3.270) {$(B_a)_b$}; \node[yshift=-2mm] at(g4.270) {$(B_{ab})_c$}; \node[yshift=-2mm] at(g5.270) {$(B_{abc})_a$};
\end{tikzpicture}
\end{center}
\caption{Local inversion of a graph with respect to the string $abca$.}
\label{fig:example8}
\end{figure}

We now state some known results that will be used in the following sections. \autoref{pro:1} means that it does not make sense to repeatedly apply local inversion on a vertex.

\begin{proposition}[\cite{sabidussi1987color}]\label{pro:1}
For any bicolored graph $B = (G, \beta)$, $B_{aa}=B$, where $a$ is any vertex in $G$.
\end{proposition}

Let $B=(G,\beta)$ be a bicolored graph with $\beta:V(G)\to\{-1,1\}$, and let $A\subseteq V(G)$. 
Define the operator $h_A$ by
\[
\bigl(h_A(\beta)\bigr)(v)=
\begin{cases}
-\beta(v), & \text{if } v\in A,\\[2pt]
\beta(v),  & \text{if } v\notin A.
\end{cases}
\]
We write $B^{A}=(G, h_A(\beta))$ for the bicolored graph obtained from $B$ by flipping the colors of all vertices in $A$ and leaving all others unchanged. 
We write 
$B^{a}$ for $B^{\{a\}}$.

\autoref{edge_reversal_lemma} states that in a bicolored graph, the endpoints of an edge can be color-reversed using $6$ local inversions, while the underlying graph remains unchanged.
 
\begin{proposition}[\cite{sabidussi1987color}]\label{edge_reversal_lemma}
Let $B=(G,\beta)$ be a bicolored graph. If $ab\in E(G)$ and $w=ababab$ (of length~$6$), then $B_w=B^{\{a,b\}}$.
\end{proposition}

\autoref{triangle_reversal_lemma} states that one vertex of a triangle in a bicolored graph can be color-reversed in $7$ moves. 
\begin{proposition}[\cite{sabidussi1987color}]\label{triangle_reversal_lemma}
Let $B=(G,\beta)$ be a bicolored graph. If $a,b,c\in V(G)$ form a triangle and $w=abacbac$ (of length~$7$), then $B_w=B^{a}$.
\end{proposition}

\autoref{k2k3p3} handles the base cases when $G$ has $2$ or $3$ vertices.

\begin{proposition}[\cite{sabidussi1987color}]\label{k2k3p3}
A bicolored graph on $K_2$ with vertices $a,b$ can be color-reversed using $2$ local inversions, for example, with the string $ab$.
A bicolored graph on a triangle $abc$ can be color-reversed using $9$ local inversions, for example, with the string $abababcac$.
Similarly, a bicolored graph on a path $P_3$ with vertices $a,b,c$, where $b$ is the degree-$2$ vertex, can be color-reversed using $9$ local inversions, for example, with the string $ababacacb$.
\end{proposition}

An \textit{odd tree} is a tree in which every vertex has odd degree.  
A \textit{perfect forest} of a graph $G$ is a spanning forest of $G$ in which each tree is an induced subgraph of $G$ and is an odd tree.

\begin{proposition}[\textbf{Perfect Forest Theorem}~\cite{caro2017two,scott2001induced}]\label{perfect_forest_theo}
Let $G$ be a connected graph on $n$ vertices, where $n$ is even. Then $G$ has a perfect forest, and such a forest can be computed in polynomial time.
\end{proposition}

\section{General bounds}\label{Section_3}
In this section, we prove \autoref{thm:main} and \autoref{transform}. We start with some simple lemmas derivable from the results by Sabidussi~\cite{sabidussi1987color}.
\autoref{color_reversal_path_ends} states that in a bicolored graph, the end points of an induced $P_3$ can be color-reversed in 8 local inversions, without changing the underlying graph.

\begin{lemma}\label{color_reversal_path_ends}
Let $B=(G,\beta)$ be a bicolored graph, and let $a,b,c\in V(G)$ with $a,b\in N_G(c)$ and $ab\notin E(G)$.  
Let $w=cabababc$ (of length~$8$). Then $B_w = B^{\{a,b\}}.$
\end{lemma}
\begin{proof}
Let $w'=ababab$ (of length~$6$). Then $w = cw'c$. 
In $G_c$, since $a,b\in N_G(c)$ and $ab\notin E(G)$, the edge $ab$ appears, so $a,b,c$ form a triangle in $G_c$.  
In $B_c$, relative to $B$, the colors of the neighbors of $c$ are flipped, while the colors of all other vertices remain unchanged. 
By \autoref{edge_reversal_lemma}, $(G_c)_{w'}=G_c$, and in $(B_c)_{w'}$ the colors of $a$ and $b$ are flipped (relative to $B_c$). 
Thus $B_{cw'}$ has underlying graph $G_c$, with the colors of all vertices in the neighborhood of $c$ flipped except those of $a$ and $b$, while the colors of the remaining vertices stay unchanged (relative to $B$). 
Finally, $(B_{cw'})_c$ has underlying graph $G_{cc}=G$ (by \autoref{pro:1}), and the colors of all neighbors of $c$ except $a$ and $b$ revert to their original values. Therefore, only the colors of $a$ and $b$ are flipped relative to $B$.\qed
\end{proof}

\autoref{p_3_end_reversal} states that an endvertex of an induced $P_3$ in a bicolored graph can be color-reversed using $7$ local inversions.

\begin{lemma}
    \label{p_3_end_reversal}
    Let $B=(G,\beta)$ be a bicolored graph, and let $a,b,c\in V(G)$ with $a,b\in N_G(c)$ and $ab\notin E(G)$.  
    Let $w = cabacba$ (of length~$7$). Then $B_w = B^a$.
\end{lemma}

\begin{proof}
    Let $w' = abacbac$ (of length~$7$). Then
    \[
        cw'c = (c)(abacbac)(c) = (cabacba)(cc) \sim cabacba = w,
    \]
    where the simplification $cc \sim \epsilon$ follows from \autoref{pro:1}.
    
    In $G_c$, since $a,b\in N_G(c)$ and $ab\notin E(G)$, the edge $ab$ appears, so $a,b,c$ form a triangle in $G_c$.  
    In $B_c$, relative to $B$, the colors of the neighbors of $c$ are flipped, while the colors of all other vertices remain unchanged. 
    By \autoref{triangle_reversal_lemma}, $(G_c)_{w'} = G_c$, and in $(B_c)_{w'}$ the color of $a$ is flipped (relative to $B_c$). 
    
    Thus, in $B_{cw'}$, the underlying graph is $G_c$, with the colors of all vertices in $N_G(c)$ flipped, except for $a$, while the colors of all other vertices remain unchanged (relative to $B$). 
    
    Finally, $(B_{cw'})_c$ has underlying graph $G_{cc} = G$ (by \autoref{pro:1}), and the colors of all neighbors of $c$ are flipped again. Therefore, only the color of $a$ is flipped relative to $B$, yielding $B_w = B^a$.
    \qed
\end{proof}

\autoref{triangle_reversal_lemma} and \autoref{p_3_end_reversal} together imply \autoref{single}, which states that only $7$ moves are sufficient to color-reverse any vertex in a nontrivial connected graph.

\begin{lemma}
    \label{single}
    Let $G$ be a connected graph on $n\geq 3$ vertices. Let $a$ be any vertex in $G$. Let $B$ be any bicolored graph of $G$. Then there is a string $w$ of length $7$ such that $B_w = B^a$.
\end{lemma}

\autoref{lem_p3_partition} states that the edges of an odd tree can be partitioned into copies of $P_3$ together with one $K_2$, with certain useful properties.

\begin{lemma}\label{lem_p3_partition}
Let $T$ be a rooted odd tree on $n$ vertices. Then there exists a partition $\mathcal{P}$ of the edges of $T$ into $(n-2)/2$ copies of $P_3$ together with one $K_2$, such that the following conditions are satisfied:
\begin{enumerate}
    \item In each $P_3$ of $\mathcal{P}$, the two endvertices are children (with respect to the given root) of its center.
    \item One of the endvertices of the $K_2$ in $\mathcal{P}$ is the root.
    \item Each vertex of $T$ is an endvertex of exactly one path in $\mathcal{P}$.
    \item The partition can be computed in polynomial time.
\end{enumerate}
\end{lemma}

\begin{figure}
\begin{center}
\begin{tikzpicture}[level 1/.style={level distance=30, sibling distance=10mm, },, level 2/.style={level distance=30, sibling distance=10mm, }, level 3/.style={level distance=30, sibling distance=30mm, }, level 4/.style={level distance=30, sibling distance=10mm, }, edge from parent/.style={draw=black}]
		\node[draw=black, shape = circle, text width = 2mm,  ] (root) {} {
			child { node[draw=black, shape=circle, text width = 2mm, ] {}
				child { node[draw=black, shape=circle, text width = 2mm, ] {}}
				child { node[draw=black, shape=circle, text width = 2mm, ] {}
				}
			}
			child { node[draw=black, shape=circle, text width = 2mm, ] {}
			}
			child { node[draw=black, shape=circle, text width = 2mm, ] {}
			}
		};

		\node[draw=black, shape = circle, text width = 2mm, right=45mm of root ] (root2) {} {
			child { node[draw=black, shape=circle, text width = 2mm, ] {}
			}
			child { node[draw=black, shape=circle, text width = 2mm, ] {}
				}
		};
		\node[draw=black, shape = circle, text width = 2mm, right=20mm of root2 ] (root3) {} {
			child { node[draw=black, shape=circle, text width = 2mm, ] {}
			}
			child { node[draw=black, shape=circle, text width = 2mm, ] {}
				}
		};
		\node[draw=black, shape = circle, text width = 2mm, right=15mm of root3 ] (root4) {} {
			child { node[draw=black, shape=circle, text width = 2mm, ] {}
			}
		};

		\node[] at(root) {$r$};
		\node[] at(root-1) {$v$};
		\node[] at(root-2) {$x$};
		\node[] at(root-3) {$y$};
		\node[] at(root-1-1) {$u$};
		\node[] at(root-1-2) {$w$};

		\node[below=3mm of root-1-1] (c1) {};

        \node[name=ar] at($(root-3)!.5!(root2-1)$){};
        \node[name=le, left=8mm of ar] {};
        \node[name=ri, right=8mm of ar] {};
        \draw[-{Stealth[scale=1.5]},] (le) to (ri) node[yshift=3mm] at($(le)!.5!(ri)$) {};

		\node[] at(root2) {$v$};
		\node[] at(root2-1) {$u$};
		\node[] at(root2-2) {$w$};

		\node[] at(root3) {$r$};
		\node[] at(root3-1) {$x$};
		\node[] at(root3-2) {$y$};

		\node[] at(root4) {$r$};
		\node[] at(root4-1) {$v$};

		\node[below=15mm of root-2] {$T$};
		\node[below=15mm of root2] {$P_3$};
		\node[below=15mm of root3] {$P_3$};
		\node[below=15mm of root4] {$K_2$};
		\node[below=25mm of root3] {$\mathcal{P}$};
			
    \end{tikzpicture}
\end{center}
\caption{ Partition of $T$ into $P_3s$ and $K_2$ }
\label{partition}
\end{figure}
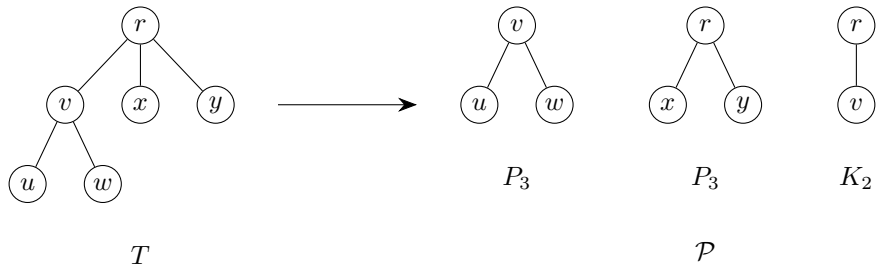
\begin{proof}

Clearly, $n$ is even. 
We proceed by induction on $n$. The base case $n = 2$ is trivial.
 
Assume $n \ge 4$ and that the claim holds for all smaller even values of $n$. 
Let $r$ be the root of $T$, and let $v$ be a vertex farthest from $r$ that is adjacent to a leaf. Then $v$ has no non-leaf descendant. Since $v$ is not a leaf and every degree in $T$ is odd, we have $\deg(v) \ge 3$. Hence, $v$ has at least two leaf neighbors; let $u$ and $w$ be two of them.

Delete the path $u\!-\!v\!-\!w$ (i.e., remove the edges $uv$ and $vw$) and then delete the isolated vertices $u$ and $w$ to obtain a tree $T'$ on $n-2$ vertices. All degrees in $T'$ remain odd (the degree of $v$ decreases by $2$, and the degrees of all other vertices are unchanged). Moreover, $T'$ has an even number of vertices. By the induction hypothesis, $E(T')$ can be partitioned into $(n-4)/2$ copies of $P_3$ together with one $K_2$. Adding the $P_3$ $u\!-\!v\!-\!w$ to this partition yields the desired partition $\mathcal{P}$. 

Since $v$ is the center of the $P_3$ $u\!-\!v\!-\!w$, statement (i) follows inductively. 
For every $P_3$ in $\mathcal{P}$ that contains $r$, by (i), $r$ must be its center. Since $r$ has odd degree, it must also be an endvertex of the $K_2$ in $\mathcal{P}$, proving (ii).
At each step, the process deletes the endvertices of the newly created $P_3$. Therefore, each vertex is the endvertex of at most one path in $\mathcal{P}$. As no vertex remains at the end, each vertex is the endvertex of exactly one path in $\mathcal{P}$, establishing statement (iii). Finally, the construction of $\mathcal{P}$ can clearly be carried out in polynomial time, proving statement (iv). See \autoref{partition} for a demonstration.
\qed
\end{proof}
 
The next two lemmas state that the vertices of an induced odd tree $T$ in a graph can be color-reversed using $4|V(T)|-4$ local inversions, by means of a string $w$ that ends (\autoref{lem_T_r_end}) or starts (\autoref{lem_T_r_start}) with a designated vertex $r$ of $T$.

\begin{lemma}\label{lem_T_r_end}
Let $G$ be a graph on $n$ vertices, and let $T$ be an induced odd tree of $G$ on at least $4$ vertices. 
Let $r$ be any vertex of $T$, and let $B = (G, \beta)$ be any bicolored graph of $G$. 
Then there exists a string $w \in V(T)^*$ such that the following conditions hold:
\begin{enumerate}
    \item $|w| = 4|V(T)|-4$,
    \item $B_w = B^{V(T)}$,
    \item $w$ ends with $r$.
\end{enumerate}
\end{lemma}

\begin{proof}
Root $T$ at $r$. By \autoref{lem_p3_partition}, there exists a partition $\mathcal{P}$ of $E(T)$ into $(n-2)/2$ copies of $P_3$ together with one $K_2$, such that in each $P_3$ the endvertices are children of its center, and every vertex of $T$ is an endvertex of exactly one path in $\mathcal{P}$. Moreover, $r$ is one of the endvertices of the $K_2$ in $\mathcal{P}$. Let $rv$ denote this $K_2$.  

Since $T$ has at least $4$ vertices, either $r$ or $v$ has degree greater than $1$ in $T$.  

First, assume that $r$ has degree greater than $1$.
Then $r$ also belongs to some $P_3$ in $\mathcal{P}$. Since $r$ can be an endvertex of at most one path in $\mathcal{P}$, it must serve as the center of every $P_3$ containing it. Let $xry$ be one such $P_3$.

Apply \autoref{color_reversal_path_ends} on $B$ successively with each $P_3$ in $\mathcal{P}$, except $xry$. Let the resulting bicolored graph be $B'$, and let $w'$ be the corresponding string used. In $B'$, every vertex of $T$ is color-flipped, except $x,y,r,$ and $v$.  

Define
\[
w_1 = vrvrvr, \quad w_2 = rxyxyxyr, \quad w'' = (vrvrv)(xyxyxyr).
\]
By \autoref{pro:1},

\noindent$ w_1w_2 = (vrvrvr)(rxyxyxyr) \;\sim\; (vrvrv)(rr)(xyxyxyr) \;\sim\; (vrvrv)(xyxyxyr)$\\$ = w''.$

By \autoref{edge_reversal_lemma}, in $B'_{w_1}$ the colors of $r$ and $v$ (with respect to $B'$) are flipped without altering the underlying graph $G$. Then, by \autoref{color_reversal_path_ends}, in $(B'_{w_1})_{w_2}$ the colors of $x$ and $y$ are flipped, again without changing $G$. Hence 
\[
B'_{w_1w_2} = B'_{w''}
\] 
is a bicolored graph in which the colors of all vertices of $T$ are flipped with respect to $B$. Note that $w = w'w''$ ends with $w''$, which in turn ends with $r$.

Now, assume that $r$ has degree $1$.  
Then $v$ has degree greater than $1$ in $T$. Let $xvy$ be a $P_3$ containing $v$ in $\mathcal{P}$. In this case, the same argument works with
\[
w_1 = vxyxyxyv, \quad w_2 = vrvrvr, \quad w'' = (vxyxyxy)(rvrvr).
\]

Finally, we count the length of $w$. We applied $w''$ with $|w''|=12$, and we used \autoref{color_reversal_path_ends} on $(|V(T)|-4)/2$ paths, contributing $8 \cdot (|V(T)|-4)/2$. Therefore,
\[
|w| = 8\cdot \frac{|V(T)|-4}{2} + 12 = 4|V(T)| - 4.
\]
This completes the proof.
\qed
\end{proof}

\autoref{lem_T_r_start}, where the only difference from \autoref{lem_T_r_end} is that $w$ starts with $r$, can be proved in a similar fashion.

\begin{lemma}\label{lem_T_r_start}
Let $G$ be a graph on $n$ vertices, and let $T$ be an induced odd tree of $G$ on at least $4$ vertices. 
Let $r$ be any vertex of $T$, and let $B = (G, \beta)$ be any bicolored graph of $G$. 
Then there exists a string $w \in V(T)^*$ such that the following conditions hold:
\begin{enumerate}
    \item $|w| = 4|V(T)|-4$,
    \item $B_w = B^{V(T)}$,
    \item $w$ starts with $r$.
\end{enumerate}
\end{lemma}

The next two lemmas state that if $G'$ is a nontrivial connected induced subgraph of a graph $G$ with $n' = |V(G')|$ even, then the vertices of $G'$ can be color-reversed (in $G$) using $4n'-4$ local inversions, by means of a string that ends (\autoref{lemma:even_v_end}) or starts (\autoref{lemma:even_v_start}) with a designated vertex $v\in V(G')$.

\begin{lemma}\label{lemma:even_v_end}
Let $G$ be a graph and let $G'$ be any connected induced subgraph of $G$ with $n' \geq 4$ vertices, where $n'$ is even. 
Let $v$ be any vertex of $G'$, and let $B = (G, \beta)$ be any bicolored graph of $G$. 
Then there exists a string $w \in V(G')^*$ such that the following conditions hold:
\begin{enumerate}
    \item $|w| \leq 4n'-4$,
    \item $B_w = B^{V(G')}$,
    \item $w$ ends with $v$.
\end{enumerate}
\end{lemma}

\begin{proof}
By \autoref{perfect_forest_theo}, $G'$ admits a perfect forest $F$. Let $T_1, T_2, \ldots, T_t$ (with $t \geq 1$) be 
the trees in $F$. Without loss of generality, assume that $v \in T_t$.

For each tree $T_i$, we construct a string $w_i$ as follows.  
If $T_i$ is a $K_2$, say $xy$, then set $w_i = xyxyxy$. In particular, if $T_t$ is a $K_2$, say $uv$, then $w_t = uvuvuv$.  
If $T_i$ has more than two vertices, then $w_i$ is the string obtained from \autoref{lem_T_r_end}. In particular, if $T_t$ is a tree with more than two vertices, then $w_t$ is the string from \autoref{lem_T_r_end}, which ends in $v$.  

Let $w = w_1w_2\cdots w_t$. Clearly, $w$ ends with $v$. By \autoref{edge_reversal_lemma} and \autoref{lem_T_r_end}, the underlying graph of $B_w$ is still $G$, and the colors of all vertices in $G'$ are flipped relative to $B$, while the colors of all other vertices remain unchanged.

It remains to prove the bound on $|w|$. Suppose $t'$ trees in $F$ are $K_2$s, for some $0 \leq t' \leq t$. Then, by \autoref{edge_reversal_lemma} and \autoref{lem_T_r_end},  
\begin{align*}
|w| &\leq 6t' + 4(n' - 2t') - 4(t - t') && \text{$t'K_2$ covers $2t'$ vertices }\\
    &= 4n' - 4t + 2t' \\
    &\leq 4n' - 4t + 2t && \text{since $t' \leq t$} \\
    &= 4n' - 2t.
\end{align*}

If $t \geq 2$, we obtain $|w| \leq 4n' - 4$ as required.  
If $t = 1$, then $G'$ itself is a tree. Since $G'$ has at least four vertices, $T_1$ cannot be a $K_2$, and the bound follows from \autoref{lem_T_r_end}.
\qed
\end{proof}

\autoref{lemma:even_v_start}, where the only difference from \autoref{lemma:even_v_end} is that $w$ starts with $v$, can be proved in a similar fashion.

\begin{lemma}\label{lemma:even_v_start}
Let $G$ be a graph and let $G'$ be any connected induced subgraph of $G$ with $n' \geq 4$ vertices, where $n'$ is even. 
Let $v$ be any vertex of $G'$, and let $B = (G, \beta)$ be any bicolored graph of $G$. 
Then there exists a string $w \in V(G')^*$ such that the following conditions hold:
\begin{enumerate}
    \item $|w| \leq 4n'-4$,
    \item $B_w = B^{V(G')}$,
    \item $w$ starts with $v$.
\end{enumerate}
\end{lemma}

\autoref{lemma:odd} states that if $G'$ is a nontrivial connected induced subgraph of $G$ with an odd number $n'$ of vertices, then the vertices of $G'$ can be color-reversed in $G$ using at most $4n'-3$ local inversions.

\begin{lemma}\label{lemma:odd}
Let $G$ be a graph and let $G'$ be any connected induced subgraph of $G$ with $n' \geq 5$ vertices, where $n'$ is odd. 
Let $B = (G, \beta)$ be any bicolored graph of $G$. 
Then there exists a string $w \in V(G')^*$ such that the following conditions hold:
\begin{enumerate}
    \item $|w| \leq 4n' - 3$,
    \item $B_w = B^{V(G')}$.
\end{enumerate}
\end{lemma}

\begin{proof}
    Let $a$ be any vertex of $G'$ such that $G'' = G' - a$ is connected. 
    
    First, assume that $a$ lies on a triangle, say $abc$ in $G'$.  
    Let $w_1 = abacbac$.  
    Applying $w_1$ to $B$ yields a bicolored graph $B_1$. By \autoref{triangle_reversal_lemma}, $B_1$ has the same underlying graph $G$, and compared to $B$, only the color of $a$ is flipped.  
    
    Since $G''$ is a connected induced subgraph of $G$ with $n'-1 \geq 4$ vertices (which is even), we can apply \autoref{lemma:even_v_start} to obtain a string $w_2$ starting with $c$ that flips the colors of all vertices of $G''$ in $B$. Therefore, $B_{w_1w_2}$ differs from $B$ exactly in that all vertices of $G'$ have their colors flipped, while the underlying graph remains $G$.  
    
    Let $w$ be the string obtained from $w_1w_2$ by removing the final $c$ of $w_1$ and the initial $c$ of $w_2$. By \autoref{pro:1}, $B_{w_1w_2} = B_w$. Furthermore,
    \begin{align*}
        |w| &= |w_1| + |w_2| - 2 \\
            &\leq 7 + (4(n'-1)-4) - 2 && \text{by \autoref{lemma:even_v_start}} \\
            &= 4n' - 3.
    \end{align*}

    Now assume that $a$ does not lie on any triangle in $G'$. 
    Since $G''$ is connected and has at least five vertices, there exist vertices $b, c \in V(G'')$ such that $acb$ is an induced $P_3$. 
    In this case, let $w_1 = cabacba$.  
    By arguments analogous to the previous case, together with \autoref{p_3_end_reversal} and \autoref{lemma:even_v_end}, we obtain the desired string $w$ satisfying the two conditions. 
    \qed
\end{proof}

We are now ready to prove \autoref{thm:main}.

\begin{theorem}
    \label{thm:main}
    Let $G$ be a connected graph with $n \geq 2$ vertices. If $n$ is even, then $cr(G) \leq 4n - 4$, and if $n$ is odd, then $cr(G) \leq 4n - 3$.
\end{theorem}

\begin{proof}
    If $n = 2$, then by \autoref{k2k3p3}, we have $cr(G)\leq 2\leq 4\cdot 2 - 4$. 
    If $n = 3$, then by \autoref{k2k3p3}, we have $cr(G) \leq 9 \leq 4 \cdot 3 - 3$.
    If $n \geq 4$ and $n$ is even, then the statement follows from \autoref{lemma:even_v_end}.  
    If $n \geq 5$ and $n$ is odd, then the statement follows from \autoref{lemma:odd}. 
    \qed
\end{proof}

\autoref{main:cor} follows immediately from \autoref{thm:main}, noting that each component of a graph with even order may itself have odd order.

\begin{corollary}
    \label{main:cor}
    Let $G$ be a graph with $n$ vertices and $t \geq 2$ components, and suppose $G$ has no isolated vertices. Then $cr(G) \leq 4n - 3t$.
\end{corollary}

We now state and prove our second main result.

\begin{theorem}
    \label{transform}
    Let $G$ be a connected graph on $n \geq 2$ vertices. 
    Let $B = (G, \beta)$ and $B' = (G, \beta')$ be two bicolored graphs on $G$. 
    Then there exists a string $w$ of length at most $ \left \lfloor \tfrac{11n-3}{2}\right \rfloor$ such that $B_w = B'$.
\end{theorem}
\begin{proof}
    Let $V_0 = \{v \in V(G) : \beta(v) = \beta'(v)\}$ with $n_0 = |V_0|$, and 
    let $V_1 = \{v \in V(G) : \beta(v) = -\beta'(v)\}$ with $n_1 = |V_1| = n - n_0$.
    If $n_1 = 0$, then $B = B'$ and we are done.
    If $n_0 = 0$, then $\beta = -\beta'$, and by \autoref{thm:main} there is a string of length at most $4n-3 \leq \frac{11n-3}{2}$ (since $n \ge 2$).

    Assume henceforth $n_0, n_1 \ge 1$.
    Let $G_0 = G[V_0]$ and $G_1 = G[V_1]$.
    Let $I_0$ (resp. $I_1$) be the set of isolated vertices in $G_0$ (resp. $G_1$), and write $i_0 = |I_0|$, $i_1 = |I_1|$.

    We consider two strategies and take the cheaper.

    \smallskip\noindent\emph{(1) Fix $V_1$ directly.}
    Use \autoref{single} on each vertex of $I_1$ (cost $7$ per isolate), then apply \autoref{thm:main} on $G_1 - I_1$ (cost at most $4(n_1 - i_1) - 3$ if $n_1 - i_1 > 0$, else $0$).
    Thus, the total cost satisfies
    \[
        p \;\le\; 7 i_1 + \max\{\,4(n_1 - i_1) - 3,\,0\,\} \;\le\; 7 n_1.
    \]

    \smallskip\noindent\emph{(2) Flip $V_0$, then flip all of $G$.}
    First use \autoref{single} on $I_0$ (cost $7$ per isolate), then apply \autoref{thm:main} on $G_0 - I_0$ (cost at most $4(n_0 - i_0) - 3$ if $n_0 - i_0 > 0$, else $0$), and finally apply \autoref{thm:main} on $G$ (cost at most $4n - 3$). 
    Hence
    \[
        q \;\le\; 7 i_0 + \max\{\,4(n_0 - i_0) - 3,\,0\,\} + (4n - 3) \;\le\; 7 n_0 + 4n - 3.
    \]

    For the given partition $(n_0,n_1)$ we can realize $B'$ from $B$ with length at most $\min\{p,q\} \le \min\{\,7(n-n_0),\; 7n_0 + 4n - 3\,\}$.
    Maximizing this upper bound over $0 \le n_0 \le n$ occurs when the two arguments are equal:
    \[
        7(n - n_0) = 7n_0 + 4n - 3 
        \;\Longrightarrow\;
        14n_0 =3n+3,
    \]
    which yields
    \[
        \min\{p,q\} \;\le\; 7\bigl(n - n_0\bigr)\;=\frac{14(n-n_0)}{2}\;=\;\frac{14n-3n-3}{2}
        \;= \;  \frac{11n - 3}{2} .
    \]
    Therefore, there exists a string $w$ of length at most $ \left \lfloor \tfrac{11n-3}{2}\right \rfloor$ such that $B_w = B'$.
    \qed
\end{proof}

\autoref{cor:transform} follows immediately.

\begin{corollary}
    \label{cor:transform}
    Let $G$ be a graph without isolated vertices, and let $B = (G, \beta)$ and $B' = (G, \beta')$ be two bicolored graphs on $G$. 
    Then there exists a string $w$ of length at most $\left \lfloor \tfrac{11n-3t}{2}\right \rfloor$ such that $B_w = B'$, where $t$ is the number of components of $G$.
\end{corollary}

\section{Stars and complete graphs}\label{Section_4}
In this section, we obtain tighter bounds for stars and complete graphs.

\begin{theorem}\label{star}
    Let $S_n$ be the star graph on $n \geq 2$ vertices. Then $cr(S_n) \leq 3n$.
\end{theorem}
\begin{proof}
    Let $c_0$ be the center and $c_1, c_2, \ldots, c_{n-1}$ be the leaves of $S_n$. 
    
    Consider the bicolored graph $B = (S_n, \beta)$. Define
    \[
        w_n = (c_1c_0c_1c_0c_1)(c_2c_0c_2)(c_3c_0c_3)\cdots(c_{n-1}c_0c_{n-1})(c_0).
    \]
    We claim that $B_{w_n} = (S_n, -\beta)$.

    For $2 \leq i \leq n$, let $A_i = \{c_0, c_1, \ldots, c_{i-1}\}$. We prove that $B_{w_i} = B^{A_i}$.
    For $i=1$, $w_1 = c_1c_0c_1c_0c_1c_0$, and this case follows from \autoref{edge_reversal_lemma}.  
    Assume that the statement holds for some $i = n-1$. We now prove it for $i = n$.

    Observe that
    \begin{align*}
        w 
            &= (w_{n-1})(c_0c_{n-1}c_0c_{n-1}c_0c_{n-1})(c_{n-1}) \hspace{5 cm}  \text{assume}\\ 
            &\sim (c_1c_0c_1c_0c_1)(c_2c_0c_2)(c_3c_0c_3)\cdots(c_{n-2}c_0c_{n-2})(c_0)
                    (c_0c_{n-1}c_0c_{n-1}c_0c_{n-1})\\ & \hspace{17pt}(c_{n-1}) \\
            &= (c_1c_0c_1c_0c_1)(c_2c_0c_2)(c_3c_0c_3)\cdots(c_{n-2}c_0c_{n-2})(c_0c_0)
                    (c_{n-1}c_0c_{n-1})(c_0)(c_{n-1}\\ & \hspace{17pt}c_{n-1}) \\
            &\sim (c_1c_0c_1c_0c_1)(c_2c_0c_2)(c_3c_0c_3)\cdots(c_{n-2}c_0c_{n-2})(c_{n-1}c_0c_{n-1})(c_0)\\
                  & \hspace{17pt}  \hspace{1pt} \text{by \autoref{pro:1}} \\
            &= w_n.  
    \end{align*}
    Therefore, it is enough to show that $B_w = (S_n, -\beta)$. 

    By the induction hypothesis, $B_{w_{n-1}} = B^{A_{n-1}}$; that is, $B_{w_{n-1}}$ has underlying graph $S_n$, where the colors of all vertices except $c_{n-1}$ are flipped.  
    Let $w' = c_0c_{n-1}c_0c_{n-1}c_0c_{n-1}$. By \autoref{edge_reversal_lemma}, $(B_{w_{n-1}})_{w'}$ has the same underlying graph $S_n$, and the colors of $c_0$ and $c_{n-1}$ are flipped relative to $B_{w_{n-1}}$.  
    Thus, $B_{w_{n-1}w'}$ has all vertex-colors flipped except $c_0$, with respect to $B$. Finally, applying $c_{n-1}$ to $(B_{w_{n-1}w'})$ preserves the underlying graph (since $c_{n-1}$ has degree $1$) and flips the color of $c_0$. Hence, all vertex colors are flipped, completing the proof.
    \qed
\end{proof}

\begin{theorem}\label{complete}
    For $n \geq 2$, we have $cr(K_n) \leq 3n$.
\end{theorem}
\begin{proof}
    Let $V(K_n) = \{c_0, c_1, \ldots, c_{n-1}\}$, and let $B = (K_n, \beta)$
    be a bicolored graph. Define
    \[
        w = (c_0)(w')(c_0),
    \]
    where
    \[
        w' = (c_1c_0c_1c_0c_1)(c_2c_0c_2)(c_3c_0c_3)\cdots(c_{n-1}c_0c_{n-1})(c_0),
    \]
    which is the string used for the color-reversal of $S_n$ in \autoref{star}. 

    In $B_{c_0}$, the underlying graph is a star with center $c_0$, and the colors of all vertices except $c_0$ are flipped.  
    By \autoref{star}, $(B_{c_0})_{w'}$ has the same underlying graph (the star centered at $c_0$), and the colors of all vertices are flipped relative to $B_{c_0}$.  
    Therefore, in $B_{c_0w'}$, only the color of $c_0$ is flipped relative to $B$.  
    Finally, applying $c_0$ once more, $(B_{c_0w'})_{c_0}$ has underlying graph $K_n$, with all vertex colors flipped relative to $B$.  

    The claim follows since
    \begin{align*}
        w &= (c_0)(w')(c_0) \\
          &= (c_0)(c_1c_0c_1c_0c_1)(c_2c_0c_2)(c_3c_0c_3)\cdots(c_{n-1}c_0c_{n-1})(c_0)(c_0) \\
          &\sim (c_0)(c_1c_0c_1c_0c_1)(c_2c_0c_2)(c_3c_0c_3)\cdots(c_{n-1}c_0c_{n-1})
          \hspace{30 pt} \text{by \autoref{pro:1}}.
    \end{align*}
    This completes the proof.\qed
\end{proof}

\section{Concluding remarks}\label{sec:conclusion}
We conclude with a few open problems.
We have shown that any bicolored graph can be color-reversed in at most $4n-4$ local inversions when $n$ is even, and in at most $4n-3$ inversions when $n$ is odd. Although this bound is tight for some small graphs, such as $P_3$, $K_3$, and $S_4$, we do not believe it is tight for larger graphs. With the aid of a computer program, we determined $cr(G)$ for all graphs with at most $5$ vertices and found that the value never exceeds $3n$. 

\noindent
\textbf{Problem 1:} Is it true that $cr(G)\leq 3n$?

We proved that Problem 1 has an affirmative answer for stars and complete graphs. But is it tight for them?

\noindent
\textbf{Problem 2:} Is it true that $cr(G) = 3n$ for all stars and complete graphs of at least 3 vertices?

As mentioned earlier, all our proofs are constructive, and the corresponding strings can be obtained in polynomial time. Nevertheless, the complexity of computing $cr(G)$ remains unknown.

\noindent
\textbf{Problem 3:} Is the following problem solvable in polynomial time?  
Given a graph $G$ and an integer $k$, decide whether $cr(G)\leq k$.

Note that no lower bounds are currently known for $cr(G)$, which is another direction worth exploring. Some good lower bounds may help us to get a good approximation algorithm. In particular:

\noindent
\textbf{Problem 4:} Does our algorithm (implied by \autoref{thm:main}) serve as an approximation algorithm with a good approximation factor for the optimization version of the problem?

The bound we obtained for the number of local inversions required to transform one bicolored graph into another (\autoref{transform}) is weaker than the color-reversal bound. However, it is unclear whether selective color-reversal inherently requires more local inversions than global color reversal.

\noindent
\textbf{Problem 5:} Is it true that the minimum number of local inversions required to transform $B$ into $B'$ is at most $cr(G)$, where $G$ is the underlying graph of $B$ and $B'$?

A trivial brute-force algorithm for the decision version runs in $n^{O(k)}$ time, implying that the problem lies in $XP$. 

\noindent
\textbf{Problem 6:} Does the problem admit an FPT algorithm parameterized by $k$?

\bibliographystyle{splncs04}
\bibliography{kumud}

\end{document}